\documentclass[12pt]{amsart}
\usepackage{amsfonts,amssymb, amscd, latexsym, graphicx, psfrag, color,float}
\usepackage{bbm}

\usepackage[all]{xy}

\usepackage{booktabs}
\usepackage{multirow}
\usepackage{mathrsfs}

\usepackage[margin=1in,marginparwidth=0.8in, marginparsep=0.1in]{geometry}

\usepackage[bookmarks=true, bookmarksopen=true,%
bookmarksdepth=3,bookmarksopenlevel=2,%
colorlinks=true,%
linkcolor=blue,%
citecolor=blue,%
filecolor=blue,%
menucolor=blue,%
urlcolor=blue]{hyperref}

\swapnumbers

\newtheorem{dummy}{dummy}[subsection]
\newtheorem{lemma}[dummy]{Lemma}

\newtheorem*{thm*}{Theorem}

\theoremstyle{definition}

\newtheorem*{prop*}{Proposition}
\newtheorem*{conj*}{Conjecture}
\newtheorem*{example*}{Example}

\numberwithin{equation}{subsection}

\newcommand{\bA}{\mathbf{A}}

\newcommand{\bC}{\mathbf{C}}

\newcommand{\bP}{\mathbf{P}}
\newcommand{\bQ}{\mathbf{Q}}
\newcommand{\bR}{\mathbf{R}}

\newcommand{\bZ}{\mathbf{Z}}

\newcommand{\cS}{\mathcal{S}}

\newcommand{\KU}{\mathbf{K}}
\newcommand{\ku}{\mathbf{ku}}
\newcommand{\cC}{\mathcal{C}}
\newcommand{\cO}{\mathcal{O}}
\newcommand{\Mod}{\mathrm{Mod}}
\newcommand{\Perf}{\mathrm{Perf}}
\newcommand{\Fuk}{\mathrm{Fuk}}
\newcommand{\Sh}{\mathrm{Sh}}

\newcommand{\Hom}{\mathrm{Hom}}
\newcommand{\Maps}{\mathrm{Maps}}

\newcommand{\Spec}{\mathrm{Spec}}
\newcommand{\Kalg}{\mathbf{K}_{\mathrm{alg}}}
\newcommand{\Kblanc}{\mathbf{K}_{\mathrm{Blanc}}}

\newcommand{\MV}{\mathrm{MV}}

\newcommand{\kmot}{\mathbf{k}_{\mathrm{mot}}}

\newcommand{\GL}{\mathrm{GL}}

\newcommand{\Spin}{\mathrm{Spin}}
\newcommand{\tors}{\mathrm{tors}}

\newcommand{\Sq}{\mathrm{Sq}}
\newcommand{\SO}{\mathrm{SO}}

\begin{document}

\title{Complex $K$-theory of mirror pairs}
\author{David Treumann}

\maketitle

\begin{abstract}
We formulate some conjectures about the $K$-theory of symplectic manifolds and their Fukaya categories, and prove some of them in very special cases.
\end{abstract}

\section{Introduction}

Let $\KU^*(X) = \KU^0(X) \oplus \KU^1(X)$ denote the complex $K$-theory of a space $X$.  I am not sure who first proposed that when $X$ and $\hat{X}$ are a mirror pair of compact Calabi-Yau $3$-folds one should have isomorphisms
\begin{equation}
\label{eq:111}
\KU^0(X) \cong \KU^1(\hat{X}) \text{ and } \KU^1(X) \cong \KU^0(\hat{X})
\end{equation}
--- it is an instance of the string-theoretical idea \cite{MM,Moore,Witten} that ``$D$-branes have charges in $K$-theory.''  Rationally, \eqref{eq:111} is a consequence of the usual Hodge-diamond flip, but the question of whether it holds becomes interesting if $\KU^*(X)$ or $\KU^*(\hat{X})$ has torsion, or if one and not the other group is known to be torsion-free.  It might be interesting more generally if one searches for very natural isomorphisms, more on that in \S\ref{sec:three}.
\medskip

I believe that \eqref{eq:111} is an open problem.  Batyrev and Kreuzer in \cite{BK} gave a case-by case verification for the half-billion mirror pairs associated with 4d reflexive polytopes, actually obtaining isomorphisms in integral cohomology
\begin{equation}
\label{eq:112}
\tors(H^2(X,\bZ)) \cong \tors(H^3(\hat{X},\bZ)) \qquad \tors(H^4(X,\bZ)) \cong \tors(H^5(\hat{X},\bZ))
\end{equation}
and deducing \eqref{eq:111} from the Atiyah-Hirzebruch spectral sequence.  But Addington \cite{Addington} has given examples of derived equivalent $3$-folds $\hat{X}$ and $\hat{X}'$ where $H^3(\hat{X},\bZ)$ and $H^3(\hat{X}',\bZ)$ have different torsion subgroups, suggesting that \eqref{eq:112} should not hold in general.

\medskip

In \S\ref{sec:two}, we will give an explicit example, by verifying \eqref{eq:111} in one new case: a $T$-dual pair of flat $3$-folds (for which homological mirror symmetry is essentially known after \cite{Abouzaid}) 
\[
X:= X_{1,5} \qquad \hat{X} := X_{2,12}
\]
with $\KU^0(X) \cong \KU^1(\hat{X})$ but $\tors(H^2(X,\bZ)) = (\bZ/4)^3$ and $\tors(H^3(\hat{X},\bZ)) = \bZ/4$.
\medskip

In \S\ref{sec:three} we will discuss conjectures --- some of mine and one of Ganatra's --- about the $K$-theory of Fukaya categories.

\section{$3$-folds}
\label{sec:two}

\subsection{The flat $3$-manifold $B$.}

Let $B$ denote the quotient of $\bR^3/\bZ^3$ by the action of $\bZ/2 \times \bZ/2$ whose three nontrivial operators are
\begin{equation}
\label{eq:Fed-Sch}
\begin{array}{ccc}
\alpha(x_1,x_2,x_3)& := & (x_1 + \frac{1}{2},-x_2 + \frac{1}{2},-x_3) \\
\beta(x_1,x_2,x_3)& :=& (-x_1+\frac{1}{2},-x_2,x_3+\frac{1}{2}) \\
\gamma(x_1,x_2,x_3)& :=& (-x_1,x_2+\frac{1}{2},-x_3 + \frac{1}{2}) 
\end{array}
\end{equation}
It is the $3$-manifold studied in \cite{HW}.  We regard it as having a basepoint at image of $0 \in \bR^3$, and as having a flat metric given by the usual dot product on $\bR^3$.  The fundamental group of $B$ is one of the Fedorov-Schoenflies crystallographic groups, with presentation \cite[Th. 3.5.5]{Wolf}
\begin{equation}
\label{eq:Wolf}
\begin{array}{ccc}
\alpha^2 = t_1 & \alpha t_2 = t_2^{-1} \alpha &  \alpha t_3 = t_3^{-1} \alpha \\
\beta t_1 = t_1^{-1} \beta & \beta t_2 = t_2^{-1} \beta &  \beta^2 = t_3 \\
\gamma t_1 = t_1^{-1} \gamma & \gamma^2 = t_2 & \gamma t_3 = t_3^{-1} \gamma
\end{array}
\end{equation}
and
\[
[t_1,t_2] = [t_2,t_3] = [t_3,t_1] =  \gamma \beta \alpha = 1
\]
The $t_1,t_2,t_3$ are translation operators on $\bR^3$.  Being flat, the holonomy group of $B$ is a representation 
\begin{equation}
\label{eq:holonomy}
\pi_1(B) \to \SO(3)
\end{equation}
Its image is isomorphic to $\bZ/2 \times \bZ/2$ (the group of diagonal matrices in $\SO(3)$).  Abelianizing \eqref{eq:Wolf} gives $H_1(B) = \bZ/4 \oplus \bZ/4$, and since $\alpha,\beta,\gamma$ are orientation-preserving we have by Poincar\'e duality
\begin{equation}
\label{eq:HiB}
H_0(B) = \bZ \qquad H_1(B) = \bZ/4 \oplus \bZ/4 \qquad H_2(B) = 0 \qquad H_3(B) = \bZ
\end{equation}

\subsection{Tri-elliptic $3$-fold $X_{0,4}$}
\label{subsec:X04}
Let $\tau_1,\tau_2,\tau_3$ be complex numbers with positive imaginary part, and put
\begin{equation}
E_i := \bC/(\bZ + \tau_i \bZ)
\end{equation}

Let $X_{0,4}$ be the quotient of $E_1 \times E_2 \times E_3$ by the complexification of the operators \eqref{eq:Fed-Sch}, i.e.
\begin{equation}
\label{eq:Fed-Sch-complex}
\begin{array}{ccc}
\alpha(z_1,z_2,z_3)& := & (z_1 + \frac{1}{2},-z_2 + \frac{1}{2},-z_3) \\
\beta(z_1,z_2,z_3)& :=& (-z_1+\frac{1}{2},-z_2,z_3+\frac{1}{2}) \\
\gamma(z_1,z_2,z_3)& :=& (-z_1,z_2+\frac{1}{2},-z_3 + \frac{1}{2}) 
\end{array}
\end{equation}
(We follow \cite{DW} for the name).  The projections $\bC \to \bR:x_i + \tau_i y_i \mapsto x_i$ descend to a map
\begin{equation}
\label{eq:X04B}
X_{0,4} \to B
\end{equation}
which is split by the subset cut out by $y_1 = y_2 = y_3 = 0$.  The translation action of
\begin{equation}
\label{eq:V}
V:= \bR\tau_1 \times \bR\tau_2 \times \bR\tau_3
\end{equation}
on $\bC \times \bC \times \bC$ descends to an action on $E_1 \times E_2 \times E_3$ and on $X_{0,4}$.  The action preserves the fibers  of \eqref{eq:X04B}, and determines an identification of the fiber over $b$ with the quotient of $V$ by a lattice $V_{\bZ,b} \subset V$.  We will denote the lattice over the basepoint by $M_{0,4}$, i.e.
\begin{equation}
\label{eq:M04}
M_{0,4} := V_{\bZ,0} = \bZ\tau_1 \times \bZ\tau_2 \times \bZ \tau_3
\end{equation}

The action of $\pi_1(B)$ on $V$ and on $M_{0,4}$ is through the holonomy $\bZ/2 \times \bZ/2$ \eqref{eq:holonomy}.

\subsection{More tri-elliptic $3$-folds}
On each $E_i$ we may define a biholomorphic action of $\bZ/2 \times \bZ/2 \times \bZ/2$: the three generators act by
\[
z \mapsto z+1/2 \qquad z \mapsto z + \tau_i/2 \qquad z \mapsto -z
\]
Altogether this defines an action of $(\bZ/2)^{\times 9}$ on $E_1 \times E_2 \times E_3$.  In \cite{DW}, Donagi and Wendland classified the subgroups that act freely.  The quotient $X = (E_1 \times E_2 \times E_3) /G$ must factor as a product of a surface and an elliptic curve, or else be isomorphic to one of the foursome
\begin{equation}
\label{eq:DW-names}
X_{0,4} \qquad X_{1,5} \qquad X_{1,11} \qquad X_{2,12}
\end{equation}
where $X_{0,4}$ is as in \S\ref{subsec:X04} and the other three are defined below.  These $3$-folds are part of a more general classification problem considered in \cite{DW}, which is reflected in the weird names.  They also appear in \cite{Lange}, where they are called ``hyperelliptic $3$-folds of type (2,2).''  Some older appearances are given in \cite{DonagiSharpe}.

Each of the $3$-folds \eqref{eq:DW-names} is aspherical, and fits into a fiber sequence
\begin{equation}
\label{eq:T-here}
V/M_{I,J} \to X_{I,J} \to B
\end{equation}
where $V$ is as in \eqref{eq:V} and $M_{I,J}$ is a lattice in $V$.

\subsection{Definition}
Let $X_{1,5}$ denote the quotient of $X_{0,4}$ by the involution
\begin{equation}
\label{eq:X15quot}
\qquad (z_1,z_2,z_3) \mapsto \left(z_1 + \frac{\tau_1}{2},z_2 + \frac{\tau_2}{2},z_3 + \frac{\tau_3}{2}\right) 
\end{equation}
Then
\[
M_{1,5} = M_{0,4} + \bZ\left(\tau_1/2, \tau_2/2, \tau_3/2\right)
\]

\subsection{Definition}
Let $X_{1,11}$ denote the quotient of $X_{0,4}$ by the involution
\begin{equation}
\label{eq:X111quot}
(z_1,z_2,z_3) \mapsto \left(z_1 + \frac{\tau_1}{2},z_2 + \frac{\tau_2}{2},z_3\right)
\end{equation}
Then
\[
M_{1,11} = M_{0,4} + \bZ\left(\tau_1/2,\tau_2/2,0\right)
\]

\subsection{Definition}
Let $X_{2,12}$ denote the quotient of $X_{0,4}$ by the $\bZ/2 \times \bZ/2$ group generated by the pair of involutions
\begin{equation}
\label{eq:X212quot}
(z_1,z_2,z_3) \mapsto \left(z_1 + \frac{\tau_1}{2},z_2 + \frac{\tau_2}{2},z_3\right) \text{ and }(z_1,z_2,z_3) \mapsto \left(z_1 ,z_2 + \frac{\tau_2}{2},z_3 + \frac{\tau_3}{2}\right)
\end{equation}
Then
\[
M_{2,12} = M_{0,4} + \bZ\left\{(\tau_1/2,\tau_2/2,0),(0,\tau_2/2,\tau_3/2)\right\}
\]

\subsection{$T$-duality}
The $T$-dual fibration to $X_{I,J} \to B$, of Strominger-Yau-Zaslow, is the space of pairs $(b,L)$ where $b \in B$ and $L \in H^1(V/V_{\bZ,b},\mathrm{U}(1))$ is the isomorphism class of a rank one unitary local system on the fiber above $b$.  Let us denote it by $\hat{X}_{I,J}$.  It is another split torus fibration
\begin{equation}
\label{eq:T-hat-here}
V^*/\hat{M}_{I,J} \to \hat{X}_{I,J} \to B
\end{equation}
where $V^* := \Hom(V,\mathfrak{u}(1))$ and $\hat{M} \subset V^*$ is the dual lattice to $M$.  As such $\hat{X}_{I,J}$ is determined up to homotopy equivalence by the dual $\pi_1(B)$-module (equivalently, the dual $\bZ/2 \times \bZ/2$-module) to $M_{I,J}$.  $M_{0,4}$ and $M_{1,11}$ are self-dual, while $M_{1,5}$ and $M_{2,12}$ are dual to each other, and therefore we have homotopy equivalences
\begin{equation}
\label{eq:T-dual-IJ}
\hat{X}_{0,4} \simeq X_{0,4} \qquad \hat{X}_{1,5} \simeq X_{2,12}, \qquad \hat{X}_{1,11} \simeq X_{1,11}
\end{equation}
The homotopy equivalences \eqref{eq:T-dual-IJ} can be taken to be natural diffeomorphisms, if $X_{I,J}$ has parameters $\tau_1,\tau_2,\tau_3$ and we take the corresponding parameters for $\hat{X}_{I,J}$ to be the purely imaginary numbers $(i|\tau_1|^{-1},i|\tau_2|^{-1}, i|\tau_3|^{-1})$.

\subsection{$K$-theory}
Let $X = X_{I,J}$ and $\hat{X} = X_{I',J'}$ be a dual pair of the $3$-folds.  We wish to prove \eqref{eq:111}, that $\KU^0(X) \cong \KU^1(\hat{X})$ and that $\KU^1(X) \cong \KU^0(\hat{X})$ --- we will do so without actually computing $\KU^*(X)$ and $\KU^*(\hat{X})$, indeed I do not quite know what the $K$-theory of these manifolds is \S\ref{subsec:HXIJ}--\ref{subsec:AH-fil}.

Let $\KU$ denote the complex $K$-theory spectrum.  It is an $E_{\infty}$-ring spectrum.  We write $\Mod(\KU)$ for the symmetric monoidal $\infty$-category of module spectra over $\KU$, and we will study sheaves of $\KU$-module spectra on $X$, $\hat{X}$ and related spaces.  These are stable $\infty$-categories --- for an $\infty$-category we will write $\Maps(c,d)$ for the space of maps and $[c,d]$ for the set of homotopy classes of maps between two objects.  We write $\Sigma$ for the suspension functor in a stable $\infty$-categories.

If $U$ is a manifold we write $\KU_U$ for the constant sheaf of $\KU$-module spectra on $U$, and $\omega_U$ for the orientation sheaf.

\subsection{Lemma}
\label{lem:spinc-structures}
Each of the spaces $B, X, \hat{X}, X \times_B \hat{X}$ are $\KU$-orientable --- that is, there are isomorphisms of sheaves
\begin{equation}
\label{eq:spin-on-these}
\Sigma^{-3} \KU_B \cong \omega_B \qquad \Sigma^{-6} \KU_X \cong \omega_X \qquad \Sigma^{-6} \KU_{\hat{X}} \cong \omega_{\hat{X}} \qquad \Sigma^{-9} \KU_{X \times_B \hat{X}} \cong \omega_{X \times_B \hat{X}}
\end{equation}

\begin{proof}
Any $\Spin^c$-structure on a manifold induces a $\KU$-orientation, and one way to endow an oriented flat manifold with a $\Spin^c$ structure is to lift its holonomy representation
\begin{equation}
\label{eq:spin-n}
\pi_1 \to \SO(n)
\end{equation}
along the natural homomorphism $\Spin^c(n) \to \SO(n)$.  Each of $B$, $X$, $\hat{X}$ and $X \times_B \hat{X}$ fibers over $B$, and the holonomy around any loop in those fibers is trivial, so \eqref{eq:spin-n} factors through $\pi_1(B) \to \SO(3)$ \eqref{eq:holonomy}.  The equations \eqref{eq:Wolf} can be solved in $\Spin^c(3)$,  for instance we may solve them in $\Spin(3)$ by taking $\alpha,\beta,\gamma$ to be the usual unit quaternions.  Then the lift of \eqref{eq:spin-n} can be taken to be the composite of $\pi_1 \to \pi_1(B) \to \Spin(3)$ with any lift of $\Spin(3) \to \SO(3) \to \SO(n)$ to $\Spin^c(n)$.
\end{proof}

\subsection{Local-on-$B$ identifications of $K$-theory}

Write $\KU[U]$ for the $K$-homology spectrum and $\KU^U$ for the $K$-cohomology spectrum of a space $U$ --- that is, $\KU[U]$ is the smash product of $\KU$ with the suspension spectrum of $U$ and $\KU^U$ is the internal mapping object from $\KU[U]$ to $\KU$.  They are related to the $K$-homology and $K$-cohomology groups of $U$ by
\[
[\Sigma^i \KU,\KU[U]] \cong \KU_i(U)
\]
and
\[
\KU^i(U) \cong [\KU[U],\Sigma^i \KU] \cong [\Sigma^{-i} \KU, \KU^U]
\]
In terms of sheaf operations, we have
\[
\Gamma_c(\omega_U) = \KU[U] \qquad \Gamma(\KU_U) = \KU^U
\]

We consider the fiber square
\begin{equation}
\label{eq:this-square}
\xymatrix{
X \times_B \hat{X} \ar[r]^-g \ar[d]_-h & X \ar[d]^-{p} \\
\hat{X} \ar[r]_-{q} & B
}
\end{equation}
Factoring the maps $X \to \mathit{pt}$ and $\hat{X} \to \mathit{pt}$ through $B$ gives canonical isomorphisms
\begin{equation}
\label{eq:from-local-on-B}
\KU[X] \cong \Gamma_c(B,p_! \omega_X) = \Gamma(B,p_!\omega_X) \qquad \KU^{\hat{X}} \cong \Gamma(B,q_* \KU_{\hat{X}})
\end{equation}
where we replace $\Gamma_c$ with $\Gamma$ using the compactness of $B$.  The $\KU$-orientability of $X$ gives an identification of $\KU[X] \cong \Sigma^{-6} \KU^{\hat{X}}$.  So to prove \eqref{eq:111} it suffices to produce an isomorphism between $\Sigma^{-3} p_! \omega_X$ and $q_* \KU_{\hat{X}}$.  To that end, let us study the sheaf of spaces on $B$ whose sections over $U \subset B$ are given by
\begin{equation}
\label{eq:sheaf-of-maps}
\Maps\left(\left(\Sigma^{-3} p_! \omega_X\right)\vert_U, \left(q_* \KU_{\hat{X}}\right)\vert_U\right)
\end{equation}
where $\Maps$ is taken in the $\infty$-category of sheaves of $\KU$-modules over $U$.  

\subsection{Lemma}
If $\pi = q\circ h = p \circ g$ denotes the projection $X \times_B \hat{X} \to B$, and one fixes $\KU$-orientations of $X$, $\hat{X}$, and $X\times_B \hat{X}$, there are natural isomorphisms
\begin{equation}
\label{eq:sheafhom2}
\Maps\left(\left(\Sigma^{-3} p_! \omega_X\right)\vert_U, \left(q_* \KU_{\hat{X}}\right)\vert_U\right)
 \cong \Maps(\KU[\pi^{-1}(U)],\KU)
\end{equation}
where the left-hand side is \eqref{eq:sheaf-of-maps} and on the right-hand side $\Maps$ is taken in the $\infty$-category of $\KU$-modules.

\begin{proof}

\begin{eqnarray}
\quad \quad \Maps\left(\left(\Sigma^{-3} p_! \omega_X\right)\vert_U, \left(q_* \KU_{\hat{X}}\right)\vert_U\right) & \cong &  \Maps\left(\left(\Sigma^{-3} q^* p_! \omega_X\right)\vert_{q^{-1}(U)}, \KU_{{q^{-1}(U)}}\right) \label{eq:364} \\
& \cong & \Maps\left(\left(\Sigma^{-3} h_! g^* \omega_X\right)\vert_{q^{-1}(U)}, \KU_{{q^{-1}(U)}}\right) 
\label{eq:365}
\\
\label{eq:366}
& \cong & \Maps\left(\Sigma^{-3} \left(h_! g^* \Sigma^{-6} \KU_{X}\right)\vert_{q^{-1}(U)}, \KU_{{q^{-1}(U)}}\right) \\
\label{eq:367}
& \cong & \Maps\left(\Sigma^{-9} \left(h_! \KU_{X \times_B \hat{X}}\right)\vert_{q^{-1}(U)}, \KU_{{q^{-1}(U)}}\right) \\
\label{eq:368}
& \cong & \Maps\left(\left(h_! \omega_{X \times_B \hat{X}}\right)\vert_{q^{-1}(U)}, \KU_{{q^{-1}(U)}}\right) \\
\label{eq:369}
& \cong & \Maps\left(\left(h_! \omega_{X \times_B \hat{X}}\right)\vert_{q^{-1}(U)}, \Sigma^{6}\omega_{{q^{-1}(U)}}\right) \\
\label{eq:3610}
& \cong & \Maps\left(\Gamma_c\left(\left(h_! \omega_{X \times_B \hat{X}}\right)\vert_{q^{-1}(U)}\right), \Sigma^{6}\KU\right) \\
\label{eq:3611}
& \cong & \Maps(\KU[(q \circ h)^{-1}(U)],\Sigma^6\KU)
\end{eqnarray}
where \eqref{eq:364} is the $(q^*,q_*)$-adjunction, \eqref{eq:365} is proper base-change, \eqref{eq:366} uses the $\KU$-orientation of $X$,\eqref{eq:368} uses the $\KU$-orientation of $X \times_B \hat{X}$,  \eqref{eq:369} uses the $\KU$-orientation of $q^{-1}(U) \subset \hat{X}$, \eqref{eq:3610} uses the $(q^{-1}(U) \to \mathit{pt})_!,(q^{-1}(U) \to \mathit{pt})^!$ adjunction.  Finally one applies the Bott isomorphism $\KU \cong \Sigma^6 \KU$ to obtain the right-hand-side of \eqref{eq:sheafhom2}.
\end{proof}

\subsection{Poincar\'e bundle}
\label{subsec:Pe-bundle}
When $T$ and $\hat{T}$ are dual tori,  (for instance, if $T = V/M$ \eqref{eq:T-here} and $\hat{T} = V^*/\hat{M}$ \eqref{eq:T-hat-here} are fibers above the basepoint of $X \to B$ and $\hat{X} \to B$), there is a canonical pairing $H_1(T) \otimes H_1(\hat{T}) \to \bZ$, which determines a canonical element
\begin{equation}
\label{eq:coev}
\mathrm{coev} \in H^1(T;\bZ) \otimes H^1(\hat{T};\bZ) \subset H^2(T \times \hat{T};\bZ)
\end{equation}
Let us say that a line bundle on $X \times_B \hat{X}$ is a ``Poincar\'e bundle'' if its restriction to a fiber is this canonical element.

The connected components of the right-hand side of \eqref{eq:sheaf-of-maps} are virtual vector bundles on $\pi^{-1}(U)$.  In particular, a line bundle on $X \times_B \hat{X}$ determines a homotopy class of maps 
\begin{equation}
\label{eq:line-bundle-map}
P_{L}:\Sigma^{-3} p_! \omega_X \to q_* \KU_{\hat{X}}
\end{equation}

\begin{lemma}
\label{lem:poincare}
If $L$ is a Poincar\'e bundle, $P_L$ is an isomorphism.
\end{lemma}

\begin{proof}
We prove that $P_L$ is an isomorphism on stalks.  More generally we prove that if $T$ and $\hat{T}$ are dual tori, a line bundle whose Chern class is \eqref{eq:coev} exhibits $\KU^T$ and $\KU^{\hat{T}}$ as dual objects in the monoidal category $\Mod(\KU)$.  Such a line bundle determines a homotopy class of maps
\begin{equation}
\label{eq:241}
\KU \to \KU^{T\times \hat{T}} \cong \KU^T \otimes_{\KU} \KU^{\hat{T}}
\end{equation}
in $\Mod(\KU)$, and we will show that for all $i$ the composite
\begin{equation}
\label{eq:251}
[\Sigma^i \KU^T,\KU] \xrightarrow{\otimes \KU^{\hat{T}}} [\Sigma^i \KU^{T} \otimes_{\KU} \KU^{\hat{T}}, \KU^{\hat{T}}] \xrightarrow{\eqref{eq:241}} [\Sigma^i \KU,\KU^{\hat{T}}]
\end{equation}
is an isomorphism.

In case $T = \hat{T} = \mathrm{U}(1)$, we have canonically $\KU^T \cong \KU \oplus \Sigma \KU$, $\KU^{\hat{T}} \cong \KU \oplus \Sigma \KU$, and 
\begin{equation}
\label{eq:TTU1}
\KU^{T \times \hat{T}} \cong \KU \oplus \Sigma \KU \oplus \Sigma \KU \oplus \Sigma^2 \KU.
\end{equation}  
Then \eqref{eq:241} is the Bott isomorphism $\KU \cong \Sigma^2 \KU$ onto the last summand of \eqref{eq:TTU1}, and one can check \eqref{eq:251} directly.

In the general case, the domain of \eqref{eq:251} is $\KU_{i}(T)$ and the codomain is $\KU^{-i}(\hat{T})$, and the square
\[
\xymatrix{
\KU_1(\mathrm{U}(1)) \otimes_{\bZ} \Hom(\mathrm{U}(1),T) \ar[r] \ar[d] & \KU_1(T) \ar[d]^{\eqref{eq:251}}\\
\KU^{-1}(\mathrm{U}(1)) \otimes_{\bZ} \Hom(\hat{T},\mathrm{U}(1)) \ar[r] & \KU^{-1}(\hat{T})
}
\]
commutes, where the left vertical arrow is \eqref{eq:251} for $T = \hat{T} = \mathrm{U}(1)$, tensored with the identification of cocharacters of $T$ with characters of $\hat{T}$.  The horizontal arrows induce graded ring isomorphisms
\begin{equation}
\label{eq:ring-iso}
\Lambda (\Hom(\mathrm{U}(1),T)) \otimes \KU^* \to \KU_*(T) \qquad \Lambda (\Hom(\hat{T},\mathrm{U}(1))) \otimes \KU^* \to \KU^*(T)
\end{equation}
where the multiplication on $\KU_*(T)$ is defined using the group structure on $T$ (the Pontrjagin product), and the ring structure on $\KU^*(\hat{T})$ is tensor product of vector bundles.  Thus we may complete the proof that \eqref{eq:251} is an isomorphism by noting that it intertwines the Pontrjagin product on $\KU_*(T)$ with the tensor product on $\KU^*(\hat{T})$.  A strong form of this is true but to make use of \eqref{eq:ring-iso} we only need to note that (letting $m:T \times T \to T$ denote the multiplication and $\Delta:\hat{T} \to \hat{T} \times \hat{T}$ the diagonal) the following two elements of $\KU^0(T \times T \times \hat{T})$ are equal:
\begin{itemize}
\item The pullback of \eqref{eq:coev} along $m \times 1:T \times T \to \hat{T} \to T \times \hat{T}$
\item The pullback of \eqref{eq:coev} $\boxtimes$ \eqref{eq:coev} along the map $T \times T \times \hat{T} \to T \times \hat{T} \times T \times \hat{T}$ that carries $(t_1,t_2,\hat{t})$ to $(t_1,\hat{t},t_2,\hat{t})$ 
\end{itemize}
In fact these are equal in $H^2(T \times T \times \hat{T};\bZ)$.  It follows that two maps from the upper left to the lower right corner of the evident square 
\[
\xymatrix{
\KU[T] \otimes_{\KU} \KU[T] \ar[d] \ar[r] & \KU^{\hat{T}} \otimes_{\KU} \KU^{\hat{T}} \ar[d] \\
\KU[T] \ar[r] & \KU^{\hat{T}}
}
\]
are homotopic, and therefore that \eqref{eq:251} is a ring homomorphism.
\end{proof}

\subsection{Theorem}
\label{th:111}
Let $X$ and $\hat{X}$ be as in \eqref{eq:T-dual-IJ}.  Then \eqref{eq:111} holds, i.e.
\[
\KU^0(X) \cong \KU^1(\hat{X}) \text{ and } \KU^1(X) \cong \KU^0(\hat{X})
\]

\begin{proof}
After \eqref{eq:from-local-on-B} and Lemma \ref{lem:poincare}, it suffices to construct a Poincar\'e bundle on $X \times_B \hat{X}$.  The fundamental group $\pi_1(B)$ acts on $H^2(T \times \hat{T};\bZ) = H^2(V/M \times V^*/\hat{M};\bZ)$, and the canonical class \eqref{eq:coev} is fixed by this action.  We will prove the existence of a Poincar\'e bundle by showing that the map
\begin{equation}
\label{eq:surj}
H^2(X \times_B \hat{X};\bZ) \to H^2(T \times \hat{T};\bZ)^{\pi_1(B)}
\end{equation} 
is a surjection.  As $X \times_B \hat{X}$ is a $K(\pi,1)$-space, the domain of \eqref{eq:surj} is isomorphic to the cohomology of the fundamental group $\pi_1(X \times_B \hat{X})$.  To prove that it is a surjection is equivalent to showing that the differentials
\[
d_2:H^2(T \times \hat{T};\bZ)^{\pi_1(B)} \to H^2(\pi_1(B); H^1(T \times \hat{T};\bZ)
\]
and
\[
d_3:\ker(d_2) \to H^3(B;\bZ)
\]
vanish, in the Serre spectral sequence of the fibration $X \times_B \hat{X} \to B$.  Let us denote this spectral sequence by ${}^{X\hat{X}}E_r^{st}$.  We similarly denote the Serre spectral sequence of $X \to B$ by ${}^X E_r^{st}$ and of $\hat{X} \to B$ by ${}^{\hat{X}} E_r^{st}$.

Since the fibration has a section, all of $H^3(B;\bZ)$ must survive to the $E_{\infty}$-page, so $d_3$ vanishes.  The codomain of $d_2$ is
\begin{equation}
\label{eq:dirsumdec}
H^2(\pi_1(B),H^1(T;\bZ)) \oplus H^2(\pi_1(B),H^1(\hat{T};\bZ))
\end{equation}
The sections of $X \to B$ and $\hat{X} \to B$ induce maps $X \to X \times_B \hat{X}$ and $\hat{X} \to X \times_B \hat{X}$ that commute with the projections to $B$, which in turn induce maps of spectral sequences
\begin{equation}
\label{eq:mapsss}
{}^X E_r^{st} \to {}^{X \hat{X}} E_r^{st} \qquad {}^{\hat{X}} E_r^{st} \to {}^{X \hat{X}} E_r^{st}
\end{equation}
The direct sum decomposition \eqref{eq:dirsumdec} is induced by \eqref{eq:mapsss} on $E_2^{21}$, thus we can complete the proof by showing that ${}^X E_2^{21}$ and ${}^{\hat{X}} E_2^{21}$ survive to the $E_{\infty}$-pages, i.e. that 
\begin{equation}
{}^X E_2^{02} \to {}^X E_2^{21} \qquad {}^{\hat{X}} E_2^{02} \to {}^{\hat{X}} E_2^{21}
\end{equation}
are both zero.  Now ${}^X E_2^{02} = H^0(\pi_1;H^2(T;\bZ)) = 0$ and $H^0(\pi_1;H^2(\hat{T};\bZ)) = 0$:  $H^2(T;\bZ) \cong M \subset V$ and $H^2(\hat{T};\bZ) \cong \hat{M} \subset V^*$ as $\pi_1$-modules, and $\pi_1$ acts on $V$ and $V^*$ without invariants ($V$ and $V^*$ split as the sum of the three nontrivial characters $\pi_1 \to \GL_1(\bR)$).  
\end{proof}

\subsection{Cohomology of $X_{I,J}$}
\label{subsec:HXIJ}
The two-vertex regular cell complex structure on $S^1$, with vertices at $0$ and $\pi$, is preserved by the action of $\bZ/2 \times \bZ/2$ generated by
\[
\theta \mapsto \theta+ \pi \qquad \theta \mapsto -\theta
\]
Each of the $3$-folds \eqref{eq:DW-names} can be written as a quotient of a torus $T^6 = S^1 \times S^1 \times S^1 \times S^1 \times S^1 \times S^1$ by the free action of an elementary abelian $2$-group that preserves the product cell structure.  The cellular cochain complex of $T^6$ is a complex of free $\bZ[G]$-modules
\begin{equation}
\label{eq:T6complex}
\bZ^{64}\to\bZ^{384} \to \bZ^{960} \to \bZ^{1280} \to \bZ^{960} \to \bZ^{384} \to \bZ^{64}
\end{equation}
of $\bZ[G]$-rank $2^6 \binom{6}{i}/|G|$ in degree $i$.  Passing to invariants gives a cochain complex for the cohomology of $T^6/G$, small enough to handle by computer --- I used sage.  Besides $H^0 = H^6 = \bZ$ and $H^1 = 0$, we have
\[
\begin{array}{|c|c|c|c|c|c|}
\hline
& H^2 & H^3 & H^4 & H^5  \\
\hline
X_{0,4} & \bZ^3 \oplus (\bZ/4)^2 \oplus (\bZ/2)^3 & \bZ^8 \oplus (\bZ/2)^3 & \bZ^3 \oplus (\bZ/2)^3 & (\bZ/4)^2 \oplus (\bZ/2)^3 \\
\hline
X_{1,5} & \bZ^3 \oplus (\bZ/4)^3 & \bZ^8 \oplus (\bZ/2)^2 & \bZ^3 \oplus (\bZ/2)^2 & (\bZ/4)^3 \\
\hline
X_{1,11}  & \bZ^3 \oplus (\bZ/4)^2 \oplus (\bZ/2)^2 & \bZ^8 \oplus (\bZ/2)^2 & \bZ^3 \oplus (\bZ/2)^2 & (\bZ/4)^2 \oplus (\bZ/2)^2  \\
\hline
X_{2,12}  & \bZ^3 \oplus (\bZ/4)^2 \oplus (\bZ/2)^2 & \bZ^8 \oplus \bZ/4 & \bZ^3 \oplus \bZ/4 & (\bZ/4)^2 \oplus (\bZ/2)^2 \\
\hline
\end{array}
\]
The top row was previously computed in \cite{BCDP}, and the $H^5$ (equivalently, $H^2$) columns in \cite{DW}.  

\subsection{Atiyah-Hirzebruch filtration}
\label{subsec:AH-fil}
Let $X$ be a connected closed manifold of real dimension $6$.  $\KU^*(X)$ carries the Atiyah-Hirzebruch filtration
\begin{equation}
\label{eq:AH}
\begin{array}{ccccccccc}
\KU^0(X) & = & F^0 \KU^0(X) & \supset & F^2 \KU^0(X) & \supset & F^4 \KU^0(X) & \supset & F^6 \KU^0(X)  \\
\KU^1(X) & = & F^1 \KU^1(X) & \supset & F^3\KU^1(X) & \supset & F^5 \KU^1(X)
\end{array}
\end{equation}
where $F^k(\KU^*(X))$ consists of those classes that vanish when restricted to any $(k-1)$-dimensional submanifold.  The associated graded pieces of this filtration are the groups at the last page of the Atiyah-Hirzebruch spectral sequence:
\[
E_2^{st} = H^s(X,\KU^t(\mathit{pt})) \implies E_{\infty}^{st} = F^{s+t} \KU^s(X)/F^{s+t + 1} \KU^s(X)
\]
If $X$ is oriented, then the spectral sequence degenerates immediately: $E_2^{st} = E_{\infty}^{st}$.   The argument is given in \cite{BrDi}  --- let us briefly repeat the argument here.  Since $\KU^t(\mathit{pt}) = 0$ for $t$ odd, all the even differentials $d_{2p}$ vanish.  In general, $d_{2p-1}$ vanishes on $H^i(X,\bZ)$ for $i \leq 2p-2$ \cite[\S 7]{Atiy}, so on a 6-dimensional complex the only possible nonvanishing differential is $d_3:H^3(X,\bZ) \to H^6(X,\bZ)$.  Even this differential must vanish if $H^6(X,\bZ)$ has no torsion \cite[\S 2.4]{AtHi}, i.e. if $X$ is orientable.  

Plausibly, whenever $X$ is a Calabi-Yau $3$-fold, or even just admits a Spin structure, the Atiyah-Hirzebruch filtration might split: that is, there might be a $\bZ/2$-graded isomorphism between
\[
\KU^0(X)\oplus \KU^1(X) \text{ and } \bigoplus H^i(X;\bZ)
\]
This is claimed in \cite{Doran}, but I believe the proof there has a gap (discussed in \S\ref{subsec:chern-classes}).  I do not know whether the filtration on $\KU^*(X_{I,J})$ splits: if it does, one could conclude Theorem \ref{th:111} directly from the computations in \S\ref{subsec:HXIJ}.  On an oriented 6-manifold one necessary and sufficient condition for the filtration of $\KU^0(X)$ to split is the existence of a function $\varphi:H^2(X;\bZ) \to H^4(X;\bZ)$ that obeys
\[
\varphi(c+c') - \varphi(c) - \varphi(c') = c \cup c'
\]
(For instance, we could take $\varphi(c) = c^2/2$ if we could divide by $2$).  The problem of computing the cup product on $H^*(X_{I,J};\bZ)$ also arose in \cite{BCDP}.  Determining this by computer is more difficult --- the problem is that, although the cup product on $H^*(T^6)$ is induced by a (noncommutative) ring structure on the cochains \eqref{eq:T6complex}, the groups $G$ do not act by ring automorphisms.  One can solve this by passing to the barycentric subdivision of $S^1$ (which induces a subdivision of $(S^1)^{\times 6}$), but the resulting chain complexes are too big to treat in a simple-minded way. 

\subsection{Chern classes}
\label{subsec:chern-classes}
A virtual vector bundle has a well-defined Chern class, giving us maps
\begin{equation}
\label{eq:chern-classes}
c_i:\KU^0(X) \to H^{2i}(X,\bZ) \qquad c_i^\Sigma:\KU^1(X) \to H^{2i -1}(X,\bZ)
\end{equation}
The second map $c_i^{\Sigma}$ is the composite of 
\[
\KU^1(X) \cong \KU^2(\Sigma X) \cong \KU^0(\Sigma X) \xrightarrow{c_i} H^{2i}(\Sigma X,\bZ) = H^{2i-1}(X,\bZ)
\]
Except for $c_0$, the functions $c_i$ of \eqref{eq:chern-classes} are not group homomorphisms, they instead obey the Cartan formula $c_n(V+W) = c_n(V) c_0(W) + c_{n-1}(V) c_1(W) + \cdots + c_0(V) c_n(W)$.  The $i$th Chern class becomes a group homomorphism on $F^{2i} \KU^0(X)$, since $c_j(E) = 0$ for any $j < i$ and $E \in F^{2i} \KU^0(X)$.  As all nontrivial cup products in $H^*(\Sigma X;\bZ)$ vanish, the Cartan formula shows that $c_i^{\Sigma}:\KU^1(X) \to H^{2i - 1}(X,\bZ)$ are group homomorphisms.

Lemma 4.1 of \cite{Doran} asserts that, when $X$ is a closed oriented $6$-manifold, the map
\begin{equation}
\label{eq:DM41}
(c_2,c_3):F^4 \KU^0(X) \to H^4(X,\bZ) \oplus H^6(X,\bZ)
\end{equation}
is an isomorphism onto
\begin{equation}
\label{eq:DM41-im}
\{(c_2,c_3) \mid \Sq^2(c_2) = c_3\} 
\end{equation}
where $\Sq^2:H^4(X,\bZ/2) \to H^6(X,\bZ/2)$ is a Steenrod operation.  Lemma 4.2 of \cite{Doran} asserts that the map
\begin{equation}
\label{eq:DM42}
(c_1,c_2,c_3):F^2\KU^0(X) \to H^2(X,\bZ) \oplus H^4(X,\bZ) \oplus H^6(X,\bZ)
\end{equation}
is an isomorphism onto
\begin{equation}
\label{eq:DM42-im}
\{(c_1,c_2,c_3) \mid \Sq^2(c_2) = c_3 + c_1 c_2 + c_1^3 \}
\end{equation}
I believe that \eqref{eq:DM41-im} is correct, but \eqref{eq:DM42-im} is not.  For example, if $X$ is the quintic $3$-fold, the virtual vector bundle $\cO(1) - \cO$ belongs to $F^2 \KU^0(X)$ and has $(c_1,c_2,c_3) = (h,0,0)$, where $h$ is the hyperplane section of $X \subset \bP^4$.  But $h^3 = 5 \in H^6(X,\bZ)$, which is nonzero in $H^6(X,\bZ/2)$.

\section{Conjectures}
\label{sec:three}

It should be possible to choose the isomorphisms \eqref{eq:111} to intertwine additional structures on $X$ and $\hat{X}$.

\subsection{$K$-homology}
In fact \eqref{eq:111} is expected for any mirror pair of Calabi-Yau manifolds of odd complex dimension.  If $X$ and $\hat{X}$ have even complex dimension, then we expect $\KU^i(X) \cong \KU^i(\hat{X})$ for $i = 0,1$.  I think the right way to organize these expectations is as an equivalence of $\KU$-module spectra:
\begin{equation}
\label{eq:spectrum}
\Sigma^{-n} \KU[X] \cong \KU^{\hat{X}}
\end{equation}
where $n$ is the complex dimension of $X$, $\Sigma$ denotes suspension, $\KU[?]$ denotes the $K$-homology spectrum and $\KU^?$ denotes the $K$-cohomology spectrum.  The $K$-homology and $K$-cohomology of a compact almost complex manifold are naturally identified, and $\KU$-theory is $2$-periodic, so \eqref{eq:spectrum} implies \eqref{eq:111} by taking homotopy groups.  Two $\KU$-module spectra are isomorphic if and only if their homotopy groups are isomorphic, so the converse is true as well.  But using $K$-homology in place of $K$-cohomology seems to go with the grain of homological mirror symmetry, in a way that we will explain.

\subsection{The large volume and large complex structure limits}
\label{subsec:lvllcsl}
For the rest of the paper we will be treating the symplectic geometry of $X$ and the complex geometry of $\hat{X}$. And we will assume that the symplectic form on $X$ has integral cohomology class $[\omega] \in H^2(X;\bZ)$.  The isomorphism class of line bundles whose Chern class is $[\omega]$ gives a unit in $\KU^0(X) := \pi_0(\KU^X)$, and (using the $\KU^X$-module structure on $\KU[X]$) a homotopy class of automorphisms of $\KU[X]$.  The corresponding homotopy class of automorphisms of $\KU^{\hat{X}}$ is a monodromy operator one obtains by putting $\hat{X}$ in a family $\hat{X}_t$, where $t$ runs through a punctured disk.

The Seidel strategy \cite{Seidel} for proving HMS is to prove it first in a limit --- one takes a hyperplane section $D$ of the line bundle on $X$, and the special fiber $\hat{X}_0$ at the center of the family $\hat{X}_t$, so that there is a mirror relationship between $X -D$ and $\hat{X}_0$.  $X-D$ is called the ``large volume limit'' and $\hat{X}_0$ is called the ``large complex structure limit'' of the mirror pair.  In such a case I conjecture (I am not sure how originally) that
\begin{equation}
\label{eq:at-limit}
\Sigma^{-n} \KU[X-D] \cong \KU^{\hat{X}_0}
\end{equation}
as $\KU$-modules.  For the noncompact $X-D$ or the singular $\hat{X}_0$, it is now necessary to pay attention to the difference between $\KU$-homology and $\KU$-cohomology.
\medskip

\noindent
{\bf Example.}  The case when $\hat{X} \subset \bC P^{n+1}$ is a degree $n+2$ hypersurface furnishes a standard example.  A mirror $X$ to $\hat{X}$ is obtained by resolving the singularities of an anticanonical hypersurface in a weighted projective $(n+1)$-space.  The limits $X -D$ and $\hat{X}_0$ can be described directly: $X -D \subset (\mathbf{C}^*)^{n+1}$ is any sufficiently generic hypersurface whose Newton polytope is the standard reflexive lattice simplex, e.g.
\begin{equation}
\label{eq:dual-ntic}
X-D := W^{-1}(0), \quad W:(x_0,\ldots,x_n) \mapsto x_0 + \cdots + x_n + \frac{1}{x_0\cdots x_n} - 1
\end{equation}
and $\hat{X}_0$ is the union of the coordinate hyperplanes
\begin{equation}
\label{eq:ntic}
\hat{X}_0 := \{[x_0,\ldots,x_n] \in \mathbf{C}P^{n+1} \mid x_0 \cdots x_n = 0\}
\end{equation}
For these examples, \eqref{eq:at-limit} can be deduced from a similar equivalence
\begin{equation}
\label{eq:LG}
\Sigma^{-n-1} \KU[(\bC^*)^{n+1},W^{-1}(0)] \cong \KU^{\mathbf{C}P^{n+1}}
\end{equation}
and from the long exact sequence of a pair.  The left-hand side of \eqref{eq:LG} denotes the $K$-homology of the pair $((\bC^*)^{n+1},W^{-1}(0))$, which has the same homotopy type as a bouquet of spheres --- one $(n+1)$-sphere for each critical point of $W$.  Note that \eqref{eq:LG} can be seen as a third variant of \eqref{eq:spectrum}, as $((\bC^*)^{n+1},W)$ is the Landau-Ginzburg mirror to projective space).

\subsection{$T$-duality}
\label{subsec:T-duality}
Homotopy classes of maps $\Sigma^{-n} \KU[X] \to \KU^{\hat{X}}$ are naturally identified with classes in the $n$th $K$-cohomology group $\KU^n(X \times \hat{X})$.  So if one wants to prove that $\Sigma^{-n} \KU[X]$ and $\KU^{\hat{X}}$ are isomorphic, one should investigate classes in $\KU^n(X \times \hat{X})$.  \S\ref{subsec:Pe-bundle} gives the example at the heart of SYZ --- a distinguished isomorphism class of line bundles on $T \times \hat{T}$ that (regarded as an element of $\KU^0(T \times \hat{T})$ induces an isomorphism
\begin{equation}
\label{eq:KUT-duality}
\KU[T] \cong \KU^{\hat{T}}
\end{equation}
when $T$ and $\hat{T}$ are dual tori.

When $X$ and $\hat{X}$ are mirror Calabi-Yaus of real dimension $2n$, fibering over the same base $B$ with dual torus fibers, this suggests that $\KU[X]$ and $\KU^{\hat{X}}$ could be identified by a virtual vector bundle on $X \times_B \hat{X}$ whose restriction to each fiber gives \eqref{eq:KUT-duality} --- a ``Poincar\'e bundle.'' The primary obstacle to doing this is that it is not clear what this virtual ``bundle'' should look like on singular fibers.  Indeed it should not be a bundle at all, but a class in $K$-homology $\KU_{3n}(X \times_B \hat{X})$ --- this group has a pushforward map to $\KU_{3n}(X \times \hat{X})$, which is isomorphic to $\KU^n(X \times \hat{X})$ using the $\KU$-orientations of $X$ and $\hat{X}$.

Even after discarding the singular fibers, or when they are just absent, there may be a Leray obstruction to finding the Poincar\'e bundle.  In the flat cases of \S\ref{sec:two}, this was simple but not exactly tautological.  At the large volume/large complex structure limit, the singular fibers can disappear, so that every fiber is a smooth torus (though the dimensions of these tori can jump); more precisely one can in some cases \cite{RSTZ} write $X - D$ as the homotopy colimit of a diagram of commutative Lie groups and homomorphisms, and $\hat{X}_0$ as the homotopy colimit of the diagram of dual groups (perhaps orbifolds), in this generality the Leray obstruction might be interesting.

As to singular fibers, it's been known for a long time what the necessary class in $\KU_{3n}$ looks like when $n = 2$, by hyperkahler rotating until $X \times_B \hat{X} \subset X \times \hat{X}$ is algebraic \cite{k3,BrMa}.  For higher even $n$, finding these Poincar\'e bundles is a more difficult algebraic geometry problem, even when the same hyperkahler techniques are available \cite{Arinkin, ADM}.  In general, especially for $n$ odd, the class in $\KU_{3n}(X \times_B \hat{X})$ cannot be algebraic; it would be interesting to describe it when $X \to B$ and $\hat{X} \to B$ are a dual pair of Gross's ``well-behaved'' singular $T^3$-fibrations \cite{Gross}.

\subsection{Blanc's invariant}
In \cite{Blanc}, Blanc showed how to compute the topological $K$-theory $\KU^Y$ of a complex algebraic variety $Y$ in a noncommutative fashion --- that is, Blanc introduced an invariant $\Kblanc(\cC) \in \Mod(\KU)$ for a $\bC$-linear dg category $\cC$, and showed
\begin{equation}
\label{eq:blanc}
\Kblanc(\Perf(Y)) \cong \KU^Y
\end{equation}
It is desirable to understand Blanc's invariant for categories arising from symplectic manifolds --- Fukaya categories and microlocal sheaf categories.  When $X$ is compact, K\"ahler with integer K\"ahler class, and Calabi-Yau, then Ganatra has conjectured that $\Kblanc(\Fuk(X))$ recovers the complex $K$-theory of $X$ whenever $\Fuk(X)$ is smooth and proper.  The last condition is motivated by results of \cite{Toen} (which state that when $Y$ is a compact complex manifold, $\Perf(Y)$ is smooth and proper if and only if $Y$ is algebraic) and the failure of \eqref{eq:blanc} for complex analytic manifolds that are not algebraic.

There is a basic problem with formulating Ganatra's conjecture precisely, or formulating any question about $\Kblanc(\Fuk(X))$ at all.  The Fukaya category of a symplectic manifold is not automatically defined over the complex numbers, but over a large Novikov field (we will call it $\mathfrak{N}$).  

\subsection{Achinger-Talpo and Blanc's invariant for $\bC((t))$-linear categories}
\label{subsec:achtal}
The $\bC$-linear structure on a dg category $\cC$ enters in Blanc's construction in an essential way, but for a compact symplectic manifold it is not usually possible to reduce the linear structure of $\Fuk(X)$ from $\mathfrak{N}$ to $\bC$.  Recent work of Achinger-Talpo, and also of Robalo and Antieau-Heller, allow for a definition of $\Kblanc(\cC)$ when $\cC$ is defined over $\bC((t))$ --- this version is adapted to Seidel's relative Fukaya category and to Ganatra's conjecture.

If $\cO \subset \bC((t))$ is the coordinate ring of an affine curve, and $Y \to \Spec(\cO)$ is a dominant map of algebraic varieties, then $\KU^{Y_a}$ has a local monodromy automorphism (call it $m$) at $t = 0$ whenever $Y_a$ is the fiber above a point $a$ close to $t = 0$.  We seek a computation of the pair $(\KU^{Y_a},m)$ that is both noncommutative and formal, in the sense that it depends only on the $\bC((t))$-linear category $\Perf(Y \times_{\cO} \bC((t)))$.  To define such a pair $(\KU^{Y_a},m)$ is equivalent to defining a $\KU$-module object of the $\infty$-category $\cS_{/S^1}$.

For any field $F$, let $\MV_F$ denote the $\infty$-category underlying the Morel-Voevodsky model structure for $\bA^1$-homotopy theory \cite[Def. 2.1]{MV}.  Let $\MV_F[(\bP^1)^{-1}]$ denote the stable $\infty$-category underlying the Morel-Voevodsky model category of motivic spectra over $F$ (\cite[Def. 5.7]{Voevodsky} or \cite[Def. 2.38]{Robalo}).  If $D$ is an $F$-linear triangulated dg category, let $\kmot(D) \in \MV_F[(\bP^1)^{-1}]$ denote the motivic refinement of the algebraic $K$-theory spectrum (as in \cite[Prop. 3.2]{AnHe}.  An embedding $F \to \bC$ induces a functor (preserving direct products and all small colimits)
\[
b^*:\MV_F \to \cS
\]
where $\cS$ denotes the $\infty$-category of spaces, and a similar functor on spectra that we will also denote by $b^*$ ($b$ for ``Betti'').  When $F = \bC$, the Blanc $K$-theory of $D$ is $\Kblanc(D) := \KU \otimes_{\mathbf{ku}} b^*\kmot(D)$, where $\ku$ denotes the connective complex $K$-theory spectrum.  

\begin{thm*}[Achinger-Talpo \cite{AchingerTalpo}]
There is a functor $\MV_{\mathbf{C}((t))} \to \cS_{/S^1}$ making the following diagram commute:
\begin{equation}
\xymatrix{
\MV_{\bC} \ar[r]^-{b^*}  \ar[d]_-{\times_{\bC} \bC((t))} & \cS \ar[d]^-{\times S^1} \\
\MV_{\bC((t))} \ar[r]_-{b_t^*} & \cS_{/S^1}
}
\end{equation}
\end{thm*}
The functor $b_t^*$ carries the Morel-Voevodsky space $\bZ \times B\mathrm{GL} \in \MV_{\bC((t))}$ \cite[p. 138]{MV} representing algebraic $K$-theory to $\bZ \times \mathrm{BU} \times S^1$.  It also carries $\bP^1$ to $S^2 \times S^1$, and so induces a map to spectra in $\cS_{/S^1}$.  Thus one can define the Blanc $K$-theory of a $\bC((t))$-linear category $\cC$ to be 
\begin{equation}
\KU \otimes_{\mathbf{ku}} b_t^*\kmot(\cC)
\end{equation}

\subsection{Doing without Blanc's invariant}
\label{subsec:dowithout}
Like any spectrum, $\KU^Y$ fits into Sullivan's arithmetic square (\cite[Prop. 3.20]{Sullivan} or \cite[Prop. 2.9]{Bousfield})
\begin{equation}
\label{eq:sullivan}
\xymatrix{
\KU^Y \ar[r] \ar[d] & \ar[d] \prod_p L_{\hat{p}} \KU^Y \\
L_{\bQ} \KU^Y \ar[r] & L_{\bQ} \prod_p  L_{\hat{p}} \KU^Y
}
\end{equation}
which is homotopy Cartesian.  Here $L_{\bQ}$ denotes the rationalization and $L_{\hat{p}}$ the $p$-completion of a spectrum.  Thomason's descent theorem shows that, when $Y$ is a complex algebraic variety, $L_{\hat{p}}\KU^Y$ can be recovered from the algebraic $K$-theory spectrum of $\Perf(Y)$:
\begin{equation}
\label{eq:thomason}
L_{\hat{p}} \KU^Y \cong L_{K(1),p} \Kalg(\Perf(Y))
\end{equation}
From this point of view, Blanc's theorem is equivalent to a ``noncommutative'' construction of $L_{\bQ} \KU^Y$ and of the map $L_{\bQ} \KU^Y \to L_{\bQ} \prod_p L_{K(1),p} \Kalg(\Perf(Y))$.  If one is merely interested in the isomorphism type of $\KU^Y$, then Thomason allows it to be recovered from $\Kalg(\Perf(Y))$ only.

If $\cC$ is linear over an algebraically closed extension of $\bC$, and $p$ is any prime, then $L_{K(1),p} \Kalg(\cC)$ is a $L_{\hat{p}} \KU$-module in a natural way.  So a weaker form of Ganatra's conjecture can be formulated without invoking any form of Blanc's construction, this way: if $X$ is a compact symplectic manifold of dimension $2n$, with a smooth and proper $\mathfrak{N}$-linear Fukaya category, then for every prime $p$ the pair of $L_{\hat{p}} \KU$-module spectra
\[
L_{K(1),p} \Kalg(\Fuk(X)) \text{ and } \Sigma^{-n} L_{\hat{p}} \KU[X]
\]
are isomorphic.  Maybe it's appropriate to call the desired equivalence of spectra a homological mirror analog of Thomason's \eqref{eq:thomason}.  

\subsection{The Euler pairings} Let $\psi^{-1}:\KU \to \KU$ denote the natural $E_{\infty}$-ring map that carries a virtual vector space to its complex conjugate.  It induces an autoequivalence on $\Mod(\KU)$, the $\infty$-category of $\KU$-modules.  

The $2n$-manifolds $X$ and $\hat{X}$ have distinguished $\KU$-orientations --- that is, there is a distinguished class in $\KU_{2n}(X)$ and in $\KU_{2n}(\hat{X})$ that maps to a generator of $\KU_{2n}(X,X - x_0)$ and of $\KU_{2n}(\hat{X},\hat{X} - x_0)$.  Denote these classes by $[X]$ and $[\hat{X}]$ --- one is determined by the complex structure on $\hat{X}$ and the other by any choice of compatible almost complex structure on $X$.  The action of the line bundle fixes $[X]$ and the action of the monodromy operator fixes $[\hat{X}]$.  They induce a further structure on $\Sigma^{-n}\KU[X]$ and $\KU^{\hat{X}}$, namely the ``Euler pairings''
\begin{equation}
(\psi^{-1} \Sigma^{-n} \KU[X]) \otimes_{\KU} \Sigma^{-n} \KU[X] \to \KU \qquad (\psi^{-1} \KU^{\hat{X}}) \otimes_{\KU} \KU^{\hat{X}} \to \Sigma^{-2n} \KU
\end{equation}
Under \eqref{eq:blanc} and the desired equivalence between $\Sigma^{-n} \KU[X]$ and the Blanc $K$-theory of $\Fuk(X)$, these maps should be induced by the Hom structures on these categories, suggesting the purely topological problem of choosing \eqref{eq:spectrum} so that the pairings match.
On $\pi_0$ this problem is closely related to Iritani's $\Gamma$-conjectures, or to the rationality question of \cite[\S 2.2.7]{KKP}.

If $M_1$ and $M_2$ are $\KU$-module spectra, write $B_n(M_1,M_2)$ for the spectrum of maps from $(\psi^{-1} M_1) \otimes M_2$ to $\Sigma^{-n} \KU$.  This is a nondegenerate symmetric bilinear spectrum-valued functor on $\Mod(\KU)$, it would be interesting to know the $L$-theory of $B_n$.

\subsection{Exact manifolds}
\label{subsec:exact-manifolds}
If $X$ is a Weinstein manifold, a version of the Fukaya category generated by exact Lagrangian submanifolds is naturally defined over any coefficient ring (not just for $\mathfrak{N}$-algebras).  The same is true for the category of sheaves with a microsupport condition (my comfort zone).  In either case the coefficient ring can be taken to be $\bC$ and one may apply Blanc's construction without worrying about the Novikov parameter.  I propose the following analogue of Ganatra's conjecture:

\begin{conj*}[Assembly]
Let $Q$ be a $d$-dimensional $\mathrm{Spin}^c$-manifold, let $\Lambda \subset T^* Q$ be a conic Lagrangian, and let $U$ be an open subset of $Q$.  Let $\Sh_{\Lambda}^w(U,\bC) \subset \Sh(U,\bC)$ be Nadler's wrapped variant \cite{N-wrapped} of the category of sheaves with microsupport in $\Lambda$.
\begin{enumerate}
\item There is a natural map
\begin{equation}
\label{eq:conj}
\Sigma^{-d} \KU[T^* U,T^* U - \Lambda] \to \Kblanc(\Sh_{\Lambda}^w(U,\bC)),
\end{equation}
that is covariantly functorial for open embeddings
\item Whenever $\Sh_{\Lambda}^w(U,\bC)$ is homologically smooth and proper, \eqref{eq:conj} is an isomorphism.
\end{enumerate}
\end{conj*} 

I expect that one can formulate a similar conjecture for the wrapped and partially wrapped Fukaya categories of a Weinstein manifold $X$ --- a natural map
\[
\Sigma^{-d} \KU^{\eta}[X,X -\Lambda]) \to \KU(\Fuk_{\Lambda}^w(X))
\]
where $\Lambda$ is the skeleton and $\eta$ is a twisting parameter, presumably trivialized on the cotangent bundle of a $\mathrm{Spin}^c$-manifold.

\subsection{String topology}
Known results on homological mirror symmetry for toric varieties \cite{Kuwagaki}, combined with computations like \eqref{eq:LG} give an indurect route to equivalences
\begin{equation}
\label{eq:cannot-hold}
\Kblanc(\Sh_\Lambda^w(Q;\mathbf{C})) \cong \Sigma^{-d} \KU[T^*Q,T^* Q - \Lambda]
\end{equation}
in some examples where $Q$ is a compact torus.  But the case where $Q$ is arbitrary and $\Lambda = Q$ is the zero section (we may call this the ``string topology case'' after \cite{Abouzaid2}) shows that \eqref{eq:cannot-hold} cannot hold in general.  Let us discuss this class of examples in more detail.

If $\Lambda \subset T^* Q$ is the zero section, then $\Sh^w_{\Lambda}(Q;\mathbf{C})$ is naturally equivalent to the category of left dg-modules over 
\begin{equation}
\label{eq:COQ}
\bC[\Omega Q] := C_*(\Omega Q;\bC)
\end{equation}
the $\bC$-valued chains on the based loop space of $Q$.  This quasi-isomorphism type of this algebra knows the rational homotopy type of $Q$, but nothing more, so one cannot expect to recover from it the $K$-theory of $Q$. 

Nevertheless, the \emph{algebraic} $K$-theory of $\bC[\Omega Q]$ is a variant of Waldhausen's $A$-theory of $Q$, and is the target of an assembly map \cite[\S 3.2]{Waldhausen}.  More generally, for any ring or ring spectrum $R$ there is a natural map
\begin{equation}
\label{eq:waldhausen}
\Kalg(R)[Q] \to \Kalg(R[\Omega Q])
\end{equation}
Letting $R$ run through $\bC$-algebras and taking realizations should produce a map $\KU[Q] \to \Kblanc(\bC[\Omega Q])$.   A $\Spin^c$-structure on $Q$ gives an identification of $\KU[Q]$ with $\KU[T^* Q, T^*Q - Q]$, the domain of \eqref{eq:conj}.

In Waldhausen's setting, the failure of the assembly map to be an isomorphism is very interesting.  When $R$ is the sphere spectrum, the cone on \eqref{eq:waldhausen} (whose codomain is called the $A$-theory of $Q$) is Hatcher's ``Whitehead spectrum''  \cite{Hatcher} that encodes the higher simple homotopy of $Q$, see \cite{Waldhausen} and Lurie's notes available at \url{math.harvard.edu/~lurie/281.html}.  When $R$ is a $\bC$-algebra, or anything else, I don't know if there is a similar interpretation.

\subsection{Speculation about the length filtration}
\label{subsec:metric}
I wonder whether one could recover the complex $K$-theory of an exact manifold from a suitable absolute version of the Fukaya category, even if this category is not homologically smooth.  (``Absolute'' means ``not relative,'' i.e. not defined over $\bC$ or $\bC((t))$ but only over the full Novikov field.)  It would require a version of Blanc's construction that treats the Novikov parameter in a more interesting way than \S\ref{subsec:achtal}--\S\ref{subsec:dowithout}, and one could hope that in this more interesting treatment the assembly map would become an isomorphism.  I will explain what I mean by making an explicit string-topology-style conjecture along these lines.  I have no evidence for it, but I will make some remarks after stating the conjecture.
\medskip

Let $Q$ be a Riemannian manifold, and let $\Omega_{q_0} Q$ be the space of rectifiable loops in $Q$ that start and end at a basepoint $q_0$.  We will treat the basepoint a little more carefully than at the end of \S\ref{subsec:exact-manifolds}, in order to make a point about it later.  The metric endows the chain algebra $\bC[\Omega_{q_0} Q]$ \eqref{eq:COQ} with an $\bR$-indexed filtration: for each $t \in \bR$ we let $F_{<t} \Omega_{q_0} Q \subset \Omega_{q_0} Q$ denote the space of loops of length less than $t$, and put
\[
F_{<t} \bC[\Omega_{q_0} Q] := \bC[F_{<t} \Omega_{q_0} Q]
\]

\begin{conj*}[Length and $K$-theory]
Let $(Q,q_0)$ and $(Q',q'_0)$ be compact, pointed Riemannian manifolds and suppose that there is a quasi-isomorphism of dg algebras
\begin{equation}
\label{eq:iso-lakt}
C_*(\Omega_{q_0} Q,\bC) \cong C_*(\Omega_{q'_0} Q',\bC)
\end{equation}
that for all $t$ carries $F_{<t} C_*(\Omega_{q_0} Q;\bC)$ quasi-isomorphically to $F_{<t} C_*(\Omega_{q'_0} Q';\bC)$ 
\begin{equation}
\label{eq:iso-lakt-filt}
C_*(\Omega_{q_0} Q;\bC) \xrightarrow{\sim} C_*(\Omega_{q'_0} Q';\bC)
\end{equation}
Then $\KU_*(Q) \cong \KU_*(Q')$.
\end{conj*}

A suitable Rees construction on the filtered dg algebra $F_{<\bullet}\bC[\Omega_{q_0} Q]$ might give an $\mathfrak{N}$-algebra that generates the absolute wrapped Fukaya category of the unit disk bundle in $T^* Q$.  The real conjecture, which I do not know how to formula precisely, is that there is a procedure similar to Blanc's for extracting a $\KU$-module from such a category, and that on the Fukaya category of the disk bundle of a Riemannian (or merely Finsler?) $Q$, it outputs the $K$-homology of $Q$.  (In particular, the notion of equivalence used in \eqref{eq:iso-lakt} is stronger than necessary: a Morita-style notion would be more appropriate.  For instance if $q_0$ and $q_1$ are different points of $Q$, there is not likely to be any quasi-isomorphism between $\bC[\Omega_{q_0} Q]$ and $\bC[\Omega_{q_1} Q]$ that preserves lengths, but the length-filtered space of paths from $q_0$ to $q_1$ could provide the Morita equivalence.)
\medskip

Let us give a reason to doubt the conjecture, followed by something more optimistic.  If $Q$ is simply-connected, one recovers $\bC[\Omega_q Q]$, up to quasi-isomorphism, as the cobar construction of the coalgebra of chains on $Q$ \cite[\S 2]{JMoore}.  The cobar construction has a natural filtration which seems to ``coarsely'' recover the legnth filtration on $\bC[\Omega_q Q]$, regardless of the metric.  Under the identification with the cobar complex of $\bC[Q]$, the loops of metric length $m$ are sandwiched between the cobars of word length $b_1 m$ and $b_2 m$, where $b_1$ and $b_2$ are constants independent of $m$.  So any way of recovering the $K$-theory of $Q$ would require knowledge of the exact numerical values of the breaks in the $\bR$-indexed filtration.

These breaks in the length filtration are a kind of homological, based version of the length spectrum of the metric.  The genuine length spectrum is known to recover the Laplace eigenvalues of $Q$, if the metric is generic \cite{DuistermaatGuillemin}.  Bergeron and Venkatesh have observed that similar spectral data can see a little bit of the homotopy type of $Q$ beyond the rational homotopy type \cite{BeVe}.  Specifically the Cheeger-Muller theorem gives a formula for the alternating product 
\begin{equation}
\label{eq:tors-prod}
\prod_i \#\mathrm{tors}H^i(Q;\bZ)^{(-1)^i}
\end{equation}
in terms of the Laplace-de Rham eigenvalues and the volumes of the images of $H^i(Q;\bZ)$ in the spaces of harmonic $i$-forms.  In another ``coarse'' sense, (perhaps a related one?) these eigenvalues are given by the Weyl law --- it is their exact numerical values that are needed to recover \eqref{eq:tors-prod}.
\medskip

\noindent
{\bf Example.}  Let $Q$ be a nontrivial $\mathrm{SU}(2)$-bundle over $S^4$.  Any degree one map $Q \to S^7$ induces a quasi-isomorphism 
\begin{equation}
\label{eq:S7example}
\bC[\Omega_{q_0} Q] \cong \bC[\Omega_{x_0} S^7].
\end{equation}
But if the Chern class of the bundle is $m \geq 2$, there is a little bit of torsion in the $K$-theory of $Q$: $\KU_1(Q) = \bZ \oplus \bZ/m$ (while $\KU_0(Q) = \bZ$, and $\KU_0(S^7) = \KU_1(S^7) = \bZ$).  The conjecture predicts that there is no metric on $Q$ for which \eqref{eq:S7example} preserves the length filtration.  The possibly spurious comparison made in the remarks above is that, since $\eqref{eq:tors-prod} = m$, the Laplace-de Rham spectra of $Q$ and of $S^7$ are never exactly the same for any choice of metrics.

\subsection*{Acknowledgements}
I thank Mohammed Abouzaid and Sheel Ganatra for sharing their ideas about the $K$-theory of Fukaya categories, Piotr Achinger and Mattias Talpo for their paper about $\bC((t))$-schemes, and Nick Addington and Ben Antieau for some corrections and other advice.  I have also benefited from discussions with Ron Donagi, Mauricio Romo, Paul Seidel, Jake Solomon, Semon Rezchikov, and Arnav Tripathy.  I was supported by NSF-DMS-1811971.

\end{document}